\documentclass[12pt]{article}
\usepackage{amsmath}
\usepackage{graphicx,psfrag,epsf}
\usepackage{enumerate}
\usepackage{natbib}
\usepackage{amsmath,amsthm,amsfonts,epsfig,amssymb,natbib,eucal,eufrak}
\usepackage{booktabs}
\usepackage{multirow}
\usepackage{siunitx}
\usepackage{qtree}
\usepackage{hyperref}

\usepackage{float}
\restylefloat{table}
\usepackage{color}

\usepackage{pdfpages}
\usepackage{wrapfig}

\usepackage{comment}

\usepackage{tikz}

\newtheorem{thm}{Theorem}

\newtheorem{result}{Result}

\newtheorem{remark}{Remark}

\newtheorem{cor}{Corollary}

\newcommand{\blind}{0}

\begin{document}

\def\spacingset#1{\renewcommand{\baselinestretch}%
{#1}\small\normalsize} \spacingset{1}


\if0\blind
{
  \title{\bf On the distribution of winners' scores in a round-robin tournament}

  \author{
  Yaakov Malinovsky
    \thanks{email: yaakovm@umbc.edu}
   \\
    Department of Mathematics and Statistics\\ University of Maryland, Baltimore County, Baltimore, MD 21250, USA\\
}
  \maketitle
} \fi

\if1\blind
{
  \bigskip
  \bigskip
  \bigskip
  \begin{center}
    {\LARGE\bf Title}
\end{center}
  \medskip
} \fi
\begin{abstract}
In a classical chess round-robin tournament, each of $n$ players wins, draws, or loses a game against each of the other $n-1$ players. A win rewards a player with 1 points, a draw with 1/2 point, and a loss with 0 points.
We are interested in the distribution of the scores associated with ranks of $n$ players
after ${\displaystyle {n \choose 2}}$ games, i.e. the distribution of the maximal score, second maximum, and so on.
The exact distribution for a general $n$ seems impossible to obtain; we obtain a limit distribution.

\end{abstract}

\noindent%

\noindent%

\spacingset{1.42} 
\section{Introduction}
In a classical chess round-robin tournament, each of $n$ players wins, draws, or loses a game against each of the other $n-1$ players. A win rewards a player with 1 points, a draw with 1/2 point, and a loss with 0 points.
Denoting by $X_{ij}$ the score of the player $i$ after the game with the player $j, j\neq i$,
in this article, we consider the following model:
\medskip

\noindent
{\it {\bf Model M}}:\\
For $i\neq j$, $X_{ij}+X_{ji}=1,\,\,X_{ij}\in \left\{0, 1/2, 1\right\}$;
we assume that all players are equally strong, i.e. $P\left(X_{ij}=1\right)=P\left(X_{ji}=1\right)$, and that the probability of a draw is the same for all games, denoted by $p=P\left(X_{ij}=1/2\right)$.
We also assume that all  ${\displaystyle {n \choose 2}}$ pairs of scores $\left(X_{12}, X_{21}\right),\ldots,\left(X_{1n}, X_{n1}\right),\ldots,$ $\left(X_{n-1,n}, X_{n,n-1}\right)$ are independent.

Let $s_i=\sum_{j=1, j\neq i}^{n}{X_{ij}}$ be a score of the player $i$ $(i=1,\ldots,n)$ after playing with $n-1$ opponents.
We use a standard notation and denote by $s_{(1)}\leq s_{(2)}\leq \ldots \leq s_{(n)}$ the order statistics of the random variables
$s_1,s_2,\ldots,s_n$, and further denote normalized scores (zero expectation and unit variance) by $s_1^{*}, s_2^{*},\ldots,s_n^{*}$
with the corresponding order statistics $s_{(1)}^{*}\leq s_{(2)}^{*}\leq \cdots \leq s_{(n)}^{*}$.

For the case where there are no draws, i.e. $X_{ij}\in \left\{0,1\right\}, X_{ij}+X_{ji}=1, p_{ij}=P(X_{ij}=1)=\frac{1}{2}$, \cite{H1963}
proved  that $$s_{(n)}^{*}-\sqrt{2\log(n-1)}\rightarrow 0$$ in probability as $n\rightarrow \infty$ (see also \cite{Moon2013}),
where $\log(x)$ is the logarithm of $x$, to base $e$.
The main step in his proof was establishing the following inequality (Lemma 1 in \cite{H1963}): $$P\left(s_1<k_1,\ldots,s_m<k_m\right)\leq P\left(s_1<k_1\right)\cdots P\left(s_m<k_m\right)$$ for any probability matrix $(p_{ij})$ and any numbers $(k_1,\ldots,k_m)$, $m\leq n$.

\cite{R2021} studied a binomial
tournament model ($X_{ij}\sim Bin (n_{ij}, p_{ij})$), and proved that Huber's type lemma holds for that model. He established bounds for
$P(s_i > max_{j\neq i} s_j)$ and for the number of wins for the winning team using the stochastic ordering property, which required the knowledge
of certain negative dependence structures of the scores.


\cite{MM2021} extended Huber's lemma to a large class of discrete distributions of $X_{ij}$ and
showed that for generalizations of round-robin tournaments, this extension implies convergence in probability of the normalized maximal score.
Model M is a particular case of such generalizations.

In this work, we are interested in the marginal distribution of the scores associated with the ranks of $n$ players
after ${\displaystyle {n \choose 2}}$ games under Model M, where rank $1$ is the winner's rank, rank $2$ is the second best, and so on. This means that we are interested in finding the marginal distribution of $s_{(i)}$.
The exact distribution for a general $n$ seems impossible to obtain; we obtain a limit distribution,
and demonstrate it with the three best scores in Model M.
Recently \cite{MR2022} proved that $s_1,\ldots,s_n$ are negatively associated (see \cite{JP1983} for the definition).
It simplifies the proof of the main result and allows all values of $p$ in the interval $[0,1)$ to be considered.

\section{Main Result}
Under Model M, we have the following properties of the scores $s_1, s_2\ldots,s_n$ that satisfy $s_1+s_2+\cdots+s_n=n(n-1)/2$:
\begin{enumerate}
\item[(a)]
$E_n=E(s_1)=(n-1)/2$,\,\,\,\, $\sigma_n=\sigma(s_1)=\sqrt{(n-1)(1-p)/4}$,
\item[(b)]
$\rho_n=corr(s_1,s_2)=-1/(n-1)$,
\item[(c)]
The random variables $s_1, s_2,\ldots,s_n$ are exchangeable for the fixed $n$.
\end{enumerate}

The normalized scores $s^{*}_1, s^{*}_2,\ldots, s^{*}_n$ are exchangeable random variables for the fixed $n$, i.e., n-exchangeable or finite exchangeable. Their distribution depends on $n$, and their correlation is a function of $n$.
Therefore, if they are a segment of the infinite sequence $s^{*}_1, s^{*}_2,\dots$, then they are not exchangeable, i.e., not infinite exchangeable.

Let ${\displaystyle I_{j}^{(n)}=I(s_{j}^{*}>x_n(t))}$,
where we choose $x_n(t)=a_n t+b_n$, where
\begin{equation}
\label{eq:ab}
a_n=(2 \log n)^{-\frac{1}{2}},\,\,\,\,\, b_n=(2\log n)^{\frac{1}{2}}-\frac{1}{2}(2 \log n)^{-\frac{1}{2}}\left(\log\log n+\log 4\pi\right).
\end{equation}
Set ${\displaystyle W_n= I_{1}^{(n)}+ I_{2}^{(n)}+\cdots+ I_{n}^{(n)}}$.

We prove the following result.
\begin{thm}
\label{eq:Th}
For ${\displaystyle p \in [0, 1)}$, a fixed value of $k$, and a fixed real $t$, $$\lim_{n\rightarrow \infty}P(W_n=k)=e^{-\lambda(t)}\frac{\lambda(t)^k}{k!},\,\,\,\lambda(t)= e^{-t}.$$
\end{thm}

\begin{proof} {(Theorem \ref{eq:Th})}
The result follows from Assertions presented below.
Set $$\pi_{i}^{(n)}=P(I_{i}^{(n)}=1),\,\,\,W_{n}=\sum_{i=1}^{n}I_{i}^{(n)},\,\,\,\lambda_n=E(W_n)=\sum_{i=1}^{n}\pi_{i}^{(n)}.$$

\noindent
{\it \bf Assertion 1}.

\begin{equation}
\tag{A1}\label{eq:A1}
d_{TV}\left({L}(W_n), Poi(\lambda_n)\right)\leq
\frac{1-e^{\lambda_n}}{\lambda_n}\left(\lambda_n-Var(W_n)\right)=\frac{1-e^{\lambda_n}}{\lambda_n}\left(\sum_{i=1}^{n}\left(\pi_{i}^{(n)}\right)^2-
\sum_{i\neq j}Cov\left(I_{i}^{(n)}, I_{j}^{(n)}\right)\right),
\end{equation}
where ${\displaystyle d_{TV}\left({L}(W_n), Poi(\lambda_n)\right)}$ is the total variation distance between
distributions of $W_n$ and Poisson distribution with mean $\lambda_n$.
\bigskip

\noindent
{\it \bf Assertion 2}.
\begin{equation}
\tag{A2}\label{eq:A2}
{\displaystyle  \pi_{1}^{(n)}=P\left(s_{1}^{*}>x_n(t)\right)\sim 1-\Phi\left(x_n(t)\right)},
\end{equation}
where ${\displaystyle c_n\sim k_n}$ means ${\displaystyle \lim_{n\rightarrow \infty} c_n/k_n=1}$.
\bigskip

\noindent
{\it \bf Assertion 3}.
\begin{equation}
\tag{A3}\label{eq:A3}
{\displaystyle  \lim_{n\rightarrow \infty}n \pi_{1}^{(n)}=\lim_{n\rightarrow \infty}n P(s_{1}^{*}>x_n(t))=\lambda(t)=e^{-t}}.
\end{equation}
\bigskip

\noindent
{\it \bf Assertion 4}.
\begin{equation}
\tag{A4}\label{eq:A4}
{\displaystyle \lim_{n\rightarrow \infty}n^2(P(s_{1}^{*}>x_n(t), s_{2}^{*}>x_n(t) ))=\lambda(t)^2=e^{-2t}.}
\end{equation}

In our case, since $s_1^{*},\ldots,s_{n}^{*}$ are identically distributed,
${\displaystyle  \sum_{i=1}^{n}\left(\pi_{i}^{(n)}\right)^2=n P(s^{*}_{1}>x_n)P(s^{*}_{1}>x_n),}$
and ${\displaystyle \sum_{i\neq j}Cov\left(I_{i}^{(n)}, I_{j}^{(n)}\right)=n(n-1)\left[P\left(s_{1}^{*}>x_n(t), s_{2}^{*}>x_n(t)\right)-P\left(s_{1}^{*}>x_n(t)\right)P\left(s_{2}^{*}>x_n(t)\right)\right]
}.$
Hence, from \eqref{eq:A2} and \eqref{eq:A3} it follows that

\begin{equation}
\tag{F1}\label{eq:F1}
{\displaystyle \lim_{n\rightarrow \infty}\sum_{i=1}^{n}\left(\pi_{i}^{(n)}\right)^2=0}.
\end{equation}

and from \eqref{eq:A3} and \eqref{eq:A4} it follows that

\begin{equation}
\tag{F2}\label{eq:F2}
{\displaystyle \lim_{n\rightarrow \infty}\sum_{i\neq j}Cov\left(I_{i}^{(n)}, I_{j}^{(n)}\right)=0}.
\end{equation}

Then, from \eqref{eq:F1} and \eqref{eq:F2} it follows that
${\displaystyle \lim_{n\rightarrow \infty} d_{TV}\left({L}(W_n), Poi(\lambda_n)\right)=0}$,
and this completes the proof of Theorem \ref{eq:Th}.
\end{proof}

\begin{proof} ({\it Assertion 1}).
\cite{MR2022} proved that $s_1,\ldots,s_n$ are negatively associated (see \cite{JP1983} for the definition).
For any $j=1,\ldots,n$, the indicator ${\displaystyle I_{j}^{(n)}}$ is an increasing function of $s_j$.
Hence, by Property 6 in \cite{JP1983}, the indicators ${\displaystyle I_{1}^{(n)},\ldots,I_{n}^{(n)}}$ are negatively associated.
Combining Theorem 2.I \citep{BHJ1992} and the Corollary 2.C.2 \citep{BHJ1992}, we obtain \eqref{eq:A1}.

\end{proof}

\begin{proof} ({\it Assertion 2}).
Follows from \cite{F1971}(p. 552-553, Theorem 2 or 3).
\end{proof}

\begin{proof} ({\it Assertion 3}).
Follows from {\it Assertion 2} combined with \cite{C1946} result on page 374 of his book.
\end{proof}

\begin{proof} ({\it Assertion 4}).
Recall that ${\displaystyle s_1=X_{12}+X_{13}+\cdots+X_{1n}}$ and ${\displaystyle s_2=X_{21}+X_{23}+\cdots+X_{2n}}$.
Hence, condition on the event ${\displaystyle X_{12}=k, k\in \left\{0, 1/2, 1\right\}}$, $s_1$ and $s_2$ are independent.
Let ${\displaystyle s_{1^{'}}=X_{13}+\cdots+X_{1n}, s_{2^{'}}=X_{23}+\cdots+X_{2n}}$ and
denote by  ${\displaystyle s_{1^{'}}^{*}, s_{2^{'}}^{*}}$ the corresponding normalized scores (zero expectation and unit variance).
We have,
\begin{align}
&
P(s_{1}^{*}>x_n(t), s_{2}^{*}>x_n(t)\,\big|\,X_{12}=k )=P(s_{1}^{*}>x_n(t)\,\big|\,X_{12}=k )
P(s_{2}^{*}>x_n(t)\,\big|\,X_{12}=k )\nonumber\\
&
=
P\left(s_{1^{'}}^{*}>x_{n-1}(t)\frac{x_{n}(t)}{x_{n-1}(t)}\sqrt{\frac{n-1}{n-2}}-\frac{\sqrt{2}(k-1/2)}{\sqrt{n-2}} \right)\nonumber\\
&
P\left(s_{2^{'}}^{*}>x_{n-1}(t)\frac{x_{n}(t)}{x_{n-1}(t)}\sqrt{\frac{n-1}{n-2}}-\frac{\sqrt{2}((1-k)-1/2)}{\sqrt{n-2}} \right)
\nonumber \\
\tag{F3}\label{eq:F3}
&
\sim
P\left(s_{1^{'}}^{*}>x_{n-1}(t)\right)
P\left(s_{2^{'}}^{*}>x_{n-1}(t)\right).
\end{align}
Combining \eqref{eq:F3} with the formula of total probability we obtain
\begin{equation*}
P(s_{1}^{*}>x_n(t), s_{2}^{*}>x_n(t))\sim P\left(s_{1^{'}}^{*}>x_{n-1}(t)\right)
P\left(s_{2^{'}}^{*}>x_{n-1}(t)\right),
\end{equation*}
and combining it with {\it Assertion 3} we obtain \eqref{eq:A4}.
\end{proof}

\begin{remark}
It remains an open problem if Theorem \ref{eq:Th} holds also for ${\displaystyle p \in (0, {1}/{3})}$.
\end{remark}

\section{Asymptotic distribution of the order statistics of the normalized scores}

An immediate consequence of Theorem \ref{eq:Th} is given below and describes the asymptotic distribution of the ordered normalized scores.

\begin{result}
\label{eq:mainR}
Suppose $p$ is fixed, ${\displaystyle p \in [0, 1)}$. Then, for a fixed $j$ and a fixed real number $t$,
$${\displaystyle
\lim_{n\rightarrow \infty}P\left(s^{*}_{(n-j)}\leq a_{n}t+b_{n}\right)=G(t)\sum_{k=0}^{j}e^{-\frac{tk}{k!}}},$$
where $a_n$ and $b_n$ are defined in \eqref{eq:ab} and $G(t)=e^{-e^{-t}}$ ("Gumbel" distribution function).
\end{result}

\begin{proof}
Result \ref{eq:mainR} follows from Theorem \ref{eq:Th},
since \\${\displaystyle P\left(s^{*}_{(n-j)}\leq x_n(t)\right)=P\left(W_n\leq j\right)}$,
and therefore
\begin{align*}
{\displaystyle \lim_{n\rightarrow \infty} P\left(s^{*}_{(n-j)}\leq x_n(t)\right)=
\lim_{n\rightarrow \infty}P\left(W_n\leq j\right)}=e^{-e^{-t}}\sum_{k=0}^j \frac{e^{-tk}}{k!}.
\end{align*}
\end{proof}

We demonstrate our results with the three best scores in Model M.

\subsection{Maximal Score}
For ${\displaystyle p \in [0, 1)}$ , we obtain  from Result \ref{eq:mainR} the following corollary.

\begin{cor}
\label{eq:cor}
\begin{align*}
&
E\left(s_{(n)}\right)\sim  \frac{n-1}{2}+\sqrt{\frac{(n-1)\log(n)(1-p)}{2}}
\\
&
+\sqrt{\frac{(n-1)(1-p)}{2\ln(n)}}\left\{\frac{\gamma}{2}-\frac{1}{4}\left(\log\log(n)+\log(4\pi)\right)\right\} \equiv \hat{E}_{(n)},             \\
&
\sigma\left(s_{(n)}\right)\sim \frac{\pi}{4\sqrt{3}}\sqrt{\frac{(n-1)(1-p)}{2\log(n)}}\equiv \hat{\sigma}_{(n)},
\end{align*}
\end{cor}

where $\gamma=0.5772156649\ldots$ is the Euler constant.

\begin{proof}
The moments under the distribution function $G$ can be obtained based on the following consideration.
If $Y_1,\ldots,Y_n$ are independent $exp(1)$ random variables, then straightforward calculation shows (see for example \cite{GS2020}):
\begin{equation}
\label{eq:e}
{\displaystyle \lim_{n\rightarrow\infty}P\left(Y_{(n)}-\log(n)\leq t\right)}=G(t),
\end{equation}
and
for r=$0,1,2,\ldots,n$, $$(n+1-r)\left(Y_{(r)}-Y_{(r-1)}\right)$$ are independent
exponential random variables with rate parameter $1$, where $Y_{(0)}$ is defined as zero.
Since
\begin{equation}
\label{eq:order}
Y_{(k)}=Y_{(1)}+\left(Y_{(2)}-Y_{(1)}\right)+\cdots+\left(Y_{(k)}-Y_{(k-1)}\right),
\end{equation}
we obtain
that $$E(Y_{(n)})=\sum_{j=1}^{n}\frac{1}{j},\,\,\,\,\,\,Var(Y_{(n)})=\sum_{j=1}^{n}\frac{1}{j^2}.$$
From ${\displaystyle \lim_{n\rightarrow \infty}\left\{\sum_{j=1}^{n}\frac{1}{j}-\log n\right\}=\gamma}$, ${\displaystyle \lim_{n\rightarrow \infty}\sum_{j=1}^{n}\frac{1}{j^2}=\frac{\pi^2}{6}}$ (see for example \cite{CR1996}), and \eqref{eq:e}, we obtain the expectation and variance under the distribution function $G$ as
${\displaystyle E_{G}=\gamma,\,\,\, Var_{G}=\frac{\pi^2}{6}}$.
Combining this with Result \ref{eq:mainR}, we have
\begin{equation}
E\left(s^{*}_{(n)}\right)\sim \gamma b_n+a_n,\,\,\, \sigma\left(s^{*}_{(n)}\right)\sim \sqrt{\frac{\pi^2}{6}} b_n.
\end{equation}
Then, upon substituting  $s^{*}_{(n)}=(s_{(n)}-E_n)/\sigma_n$, Corollary \ref{eq:cor} follows.
\end{proof}

In the following table, we compare $E\left(s_{(n)}\right)$ with $\hat{E}_{(n)}$ and $\sigma\left(s_{(n)}\right)$ with $\hat{\sigma}_{(n)}$ in this manner: We fix $p=2/3$ and for n=10, 20, 50, 100, 1000, and 10000 we evaluate
$E\left(s_{(n)}\right)$ and $\sigma\left(s_{(n)}\right)$ using Monte-Carlo (MC) simulation.
Values  of $\hat{E}_{(n)}$ and $\hat{\sigma}_{(n)}$ obtained based on Corollary \ref{eq:cor}.

\begin{table}[H]
\caption{
The number of Monte-Carlo repetitions is 100,000 for n=10, 20, 50, 100; 10,000 for n=1000; and 500 for n=10,000. }
\label{t:1}
\small
\begin{center}
\begin{tabular}{lllllll}
  n & $E\left(s_{(n)}\right)$ & $\hat{E}_{(n)}$ & $|\hat{E}_{(n)}/E\left(s_{(n)}\right)-1|*100\%$ &$\sigma\left(s_{(n)}\right)$&$\hat{\sigma}_{(n)}$& $|\hat{\sigma}_{(n)}/\sigma\left(s_{(n)}\right)-1|*100\%$  \\
  \hline
  10 & 5.833 & 5.912      &  1.360&  0.469&  0.518 & 10.454\\
  20 & 11.89 &  11.944    &  0.456&  0.627&  0.659 & 5.189\\
  50 & 29.08 &  29.162    &  0.283&  0.912&  0.927 & 1.563 \\
  100 & 56.73 & 56.843    &  0.199&  1.219&  1.214 & 0.426\\
  1,000 & 529.12 & 529.352&  0.044&  3.259&  3.148 & 3.529\\
  10,000&  5110.23&5111.295&0.0212&  8.949&  8.626 & 3.742\\
\end{tabular}
\end{center}
  \end{table}

\subsection{Second and third largest scores}

For ${\displaystyle p \in [0, 1)}$ , we also obtain from Result \ref{eq:mainR} the following corollary.

\begin{cor}

\begin{align}
\label{eq:m}
&
E\left(s^{*}_{(n-1)}\right)\sim \gamma b_n+a_n-b_n,\,\,\,\,\sigma\left(s^{*}_{(n-1)}\right)\sim \sqrt{\left(\frac{\pi^2}{6}-1\right)}\,b_n,\\
&
E\left(s^{*}_{(n-2)}\right)\sim \gamma b_n+a_n-3/2b_n,\,\,\,\,\sigma\left(s^{*}_{(n-2)}\right)\sim \sqrt{\left(\frac{\pi^2}{6}-1.25\right)}\,b_n
\end{align}

\end{cor}

\begin{proof}
From Theorem 2.2.2 in \cite{LLR1983}, we obtain the following result:
if $Y_1,\ldots,Y_n$ are independent $exp(1)$ random variables, then for $j=1,2$
\begin{equation}
\label{eq:ee}
{\displaystyle \lim_{n\rightarrow\infty}P\left(Y_{(n-j)}-\log(n)\leq t\right)=G(t)\left(1+e^{-t}/1!+\cdots+e^{-jt}/j!\right)}.
\end{equation}
The rest of the proof is similar to the proof of Corollary 1.
\end{proof}

Substituting $s^{*}_{(j)}=(s_{(j)}-E_n)/\sigma_n$ for $j=n-1, n-2$, we obtain the values
$E(s_{(j)}), \sigma(s_{(j)})), \widehat{E}_{(j)}, \widehat{\sigma}_{(j)}$, which are
similar to the corresponding values obtained in Corollary 1 for the case $j=n$.
In the case where $p=2/3$, we provide numerical comparisons for the second and third largest scores in a similar manner as was done in Table 1.

\begin{table}[H]
\caption{
The number of Monte-Carlo repetitions is 100,000 for n=10, 20, 50, 100; 10,000 for n=1000; and 500 for n=10,000;
${\displaystyle r_{(j)}=|\widehat{E}(s_{(j)})}/E(s_{(j)})-1|*100\%,\,\,\,j=n-1, n-2.$ }
\label{t:1}
\small
\begin{center}
\begin{tabular}{lllllllll}
  n
  & $E(s_{(n-1)})(\sigma(s_{(n-1)}))$
  & $\widehat{E}_{(n-1)}(\widehat{\sigma}_{(n-1)})$
   & $E(s_{(n-2)})(\sigma(s_{(n-2)}))$
  & $\widehat{E}_{(n-2)}(\widehat{\sigma}_{(n-2)})$
  &$r_{(n-1)}$
  &$r_{(n-2)}$\\
  \hline
  10 &5.400(0.338)      &5.509(0.324)    &5.093(0.273)    &5.307(0.254)    &2.009  &4.195\\
  20 &11.305(0.446)     &11.43(0.413)    &10.95(0.374)    &11.173(0.323)   &1.106  &2.037\\
  50 &28.277(0.649)     &28.44(0.580)    &27.816(0.541)   &28.079(0.454)   &0.576  & 0.946\\
  100 &55.695(0.858)    &55.896(0.760)   &55.113(0.712    &55.423(0.595)   &0.361  & 0.563\\
  1,000 &526.48(2.154)  &526.9(1.971)    &525.05(1.764)   &525.67(1.543)   &0.080  & 0.118\\
  10,000&5103.2(5.866)  &5104.6(5.401)   &5099.5(4.672)   &5101.2(4.227)   &0.027  &0.033\\
\end{tabular}
\end{center}
  \end{table}

\section*{Acknowledgement}
I thank Abram Kagan for describing a score issue in chess round-robin tournaments with draws.
I am grateful to Pavel Chigansky for pointing out my mistake in \cite{M2021a} (see also \cite{M2021b}).
I thank Sheldon Ross for referring to his recent paper. I also thank Yosi Rinott for the discussions and comments.
This research is supported in part by BSF grant 2020063.
{}


\begin{thebibliography}{}

\bibitem[\protect\citeauthoryear{Barbour \it{et~al.}}{1992}]{BHJ1992}
Barbour, A.~D., Holst, L., Janson, S. (1992).
\newblock Poisson approximation.
\newblock {Oxford Studies in Probability. The Clarendon Press, Oxford, New York}.




\bibitem[\protect\citeauthoryear{Cram\'{e}r}{1946}]{C1946}
Cram\'{e}r, H. (1946).
\newblock Mathematical Methods of Statistics.
\newblock {Primceton University Press}.


\bibitem[\protect\citeauthoryear{Courant and Robbins}{1996}]{CR1996}
Courant, R., and Robbins, H. (1996).
\newblock What Is Mathematics? An Elementary Approach to Ideas and Methods.
\newblock {Oxford University Press} 2nd edition reviewed by Ian Stewart.

\bibitem[\protect\citeauthoryear{Feller}{1971}]{F1971}
Feller, W. (1971).
\newblock An introduction to probabilty theory and its applications. Vol. II. Second edition.
\newblock New York-London-Sydney: Wiley.




\bibitem[\protect\citeauthoryear{Grimmett and Stirzaker}{2020}]{GS2020}
Grimmett, G.~R. and Stirzaker, D~R. (2020).
\newblock Probability and random processes. 4th Edition.
\newblock {Oxford University Press}.



\bibitem[\protect\citeauthoryear{Huber}{1963}]{H1963}
Huber, P.~J. (1963).
\newblock A remark on a paper of Trawinski and David entitled: Selection
of the best treatment in a paired comparison experiment.
\newblock {Ann.Math.Statist.} {\textbf 34,} 92--94.

\bibitem[\protect\citeauthoryear{Joag-Dev and Proschan}{1983}]{JP1983}
Joag-Dev, K., Proschan, F. (1983).
\newblock Negative association of random variables, with applications.
\newblock {Ann. Statist.} {\textbf 11,} 286--295.



\bibitem[\protect\citeauthoryear{Leadbetter \it{et~al.}}{1983}]{LLR1983}
Leadbetter, M.~R., Lindgren, G., Rootz\'{e}n, H. (1983).
\newblock Extremes and related properties of random sequences and processes.
\newblock {Springer Series in Statistics. Springer-Verlag, New York-Berlin}.



\bibitem[\protect\citeauthoryear{Malinovsky}{2022a}]{M2021a}
Malinovsky, Y.(2022a).
\newblock On the distribution of winners' scores in a round-robin tournament.
\newblock{Prob. in Eng. and Inf. Sciences} {\textbf 36,}  1098--1102.


\bibitem[\protect\citeauthoryear{Malinovsky}{2022b}]{M2021b}
Malinovsky, Y. (2022b).
\newblock Correction to "On the distribution of winners' scores in a round-robin tournament."
\newblock{Prob. in Eng. and Inf. Sciences}. {\textbf 37}, 737--739.

\bibitem[\protect\citeauthoryear{Malinovsky and Rinott}{2023}]{MR2022}
Malinovsky, Y., Rinott, Y. (2023).
\newblock On tournaments and negative dependence.
\newblock {J. Appl. Probab.} {\textbf 60,} 945--954.


\bibitem[\protect\citeauthoryear{Malinovsky and Moon}{2021}]{MM2021}
Malinovsky, Y., Moon, J.~W. (2021).
\newblock On the negative dependence inequalities and maximal score in round-robin tournament.
\newblock{https://arxiv.org/abs/2104.01450}.

\bibitem[\protect\citeauthoryear{Moon}{2013}]{Moon2013}
Moon, J.~W. (2013).
\newblock Topics on Tournaments.
[Publicaly available on website of Project Gutenberg \url{https://www.gutenberg.org/ebooks/42833}].


\bibitem[\protect\citeauthoryear{Ross}{2022}]{R2021}
Ross, S.~M. (2022).
\newblock Team's seasonal win probabilities.
\newblock {Prob. in Eng. and Inf. Sciences}.
{\textbf 36,}  988--998.









\end{thebibliography}
\end{document}